\documentclass[12pt,a4]{article}

\usepackage[square]{natbib}
\usepackage{amssymb} 
\usepackage{phonetic} 
\usepackage{amsmath}
\usepackage{algorithm}
\usepackage{algpseudocode}
\newtheorem{theorem}{Theorem}
\newtheorem{lemma}{Lemma}

\newtheorem{definition}{Definition}
\newtheorem{example}{Example}
\newtheorem{example*}{Example*}
\newtheorem{proposition}{Proposition}
\newtheorem{remark}{Remark}
\newenvironment{proof}[1][Proof]{\textbf{#1.} }{\ \rule{0.5em}{0.5em}}

\begin{document}
% --- Author Metadata here ---
%\conferenceinfo{ISSAC}{'14 Kobe, Japan}
%\CopyrightYear{2007} % Allows default copyright year (20XX) to be over-ridden - IF NEED BE.
%\crdata{0-12345-67-8/90/01}  % Allows default copyright data (0-89791-88-6/97/05) to be over-ridden - IF NEED BE.
% --- End of Author Metadata ---
\title{On the Reduction of Singularly-Perturbed Linear Differential Systems\footnote{Submitted to ISSAC'14, Kobe, Japan}}
\author {\textbf{Moulay Barkatou, Suzy S. Maddah} \footnote{Enrolled  under a joint PhD program with the Lebanese University}\\
     XLIM UMR 7252 ; DMI\\
          University of Limoges; CNRS\\
       123, Avenue Albert Thomas\\
       87060 Limoges, France\\
       moulay.barkatou@unilim.fr\\
       suzy.maddah@etu.unilim.fr\\[10pt]
         \textbf{Hassan Abbas}\\
       Laboratory of Mathematics\\
       Lebanese University, Beirut, Lebanon\\
       habbas@ul.edu.lb\\}
\maketitle
\begin{abstract}
In this article, we recover singularly-perturbed linear differential systems from their \textit{turning points} and reduce their parameter singularity's rank to its minimal integer value. Our treatment is Moser-based; that is to say it is based on the reduction criterion introduced for linear singular differential systems in \citep{key19}. Such algorithms have proved their utility in the symbolic resolution of the systems of linear functional equations \citep{key80,key64,key26}, giving rise to the package ISOLDE \citep{key27}, as well as in the perturbed algebraic eigenvalue problem \citep{key54}. In particular, we generalize the Moser-based algorithm described in \citep{key41}. Our algorithm, implemented in the computer algebra system Maple, paves the way for efficient symbolic resolution of singularly-perturbed linear differential systems as well as further applications of Moser-based reduction over bivariate (differential) fields \citep{key101}.
\end{abstract}

\section{Introduction}
Let $K$ be a commutative field of characteristic zero equipped with a derivation $\delta$, that is a map $\delta : K \rightarrow K$ satisfying
$$\delta(f + g) = \delta f + \delta g \; \text{and} \; \delta(f g) = \delta(f) g + f \delta(g) \; \text{for all} \; f, g \in K.$$
We denote by $\mathcal{C}$ its field of constants, $V$ a  $K$-vector space of dimension n, and $A$ an element in $K^{n \times n}$. Set $\Delta = \delta - A$, then $\Delta$ is a $\delta$- differential operator acting on $V$, that is, a $\mathcal{C}$-linear endomorphism of $V$ satisfying the Leibniz condition:
$$ \forall f \in K, \; v \in V   \;  \Delta (f v) = \delta (f)  v + f \Delta(v) .$$
Let $\mathcal{O}=\mathbb{C}[[x,\epsilon]]$, the ring of formal power series in $x$ and $\epsilon$ over the field of complex numbers, where $x$ is a complex variable and $\epsilon$ is a small parameter. Denote by $K$ its field of fractions equipped with $\delta = \epsilon \frac{d}{dx}$. Let $\mathcal{O}_x=\mathbb{C}[[x]]$, $K_x$ be its fields of fractions, and $h,p$ be integers ($h \geq 0$). Thereby, $\Delta$ is the differential operator associated to a singularly-perturbed linear differential system. Denoting by $Y$ an unknown $n$-dimensional column vector, the former is written as
\begin{equation}
\label{general}
\epsilon \frac{dY}{dx} = A(x, \epsilon) Y = \epsilon^{- h} x^{- p} \sum_{k=0}^{\infty} A_k(x) {\epsilon}^k  Y.
\end{equation} 
Such systems have countless applications which are traced back to the year 1817 and their study encompasses a vast body of literature (see, e.g. \citep{key61, key72, key82, key60} and references therein). However, their symbolic resolution is still open to investigation. 

Clearly, system \eqref{general} is a singular perturbation of the widely studied linear singular system of  differential equations (see, e.g., \citep{key6,key71} and references therein). The latter is obtained by setting $\delta= x \frac{d}{dx}$ and $K=\mathbb{C}((x))$, the univariate field of  formal Laurent power series. Hence $\Delta$ is associated to
\begin{equation}
 \label{particular}
x \frac{dY}{dx} = A(x)  Y = x^{-p} \sum_{k=0}^{\infty} A_k  x^k  Y.
 \end{equation}
Contrary to system \eqref{general}, there exist efficient algorithms contributing to the formal reduction of system \eqref{particular} (see, e.g \citep{key24, key33}), that is the algorithmic procedure that computes a change of basis w.r.t. which $A(x)$ has normal form facilitating the construction of formal solutions. Among which, rank reduction algorithms play an eminent role as they reduce the nonnegative integer $p$, called the \textit{Poincar\'e rank}, to its minimal integer value, the \textit{true Poincar\'e rank} (see, e.g., \citep{key26,key57}). In particular, if the latter is null then system \eqref{particular} is said to be \textit{regular singular}, that is to say, in any small sector, its solutions grow at most as an algebraic function. Otherwise, it is \textit{irregular singular}. 

However, the classical formal simplification of system \eqref{particular} begins with the reduction of its leading coefficient matrix $A_0$ to its Jordan form. Hence, in addition to the usual difficulties encountered within the formal reduction itself, additional ones arise for system \eqref{general} since its leading coefficient matrix $A_0(x)$  is a matrix-valued function rather than a constant one. In particular, if it has mulitple eigenvalues then they might coalesce (see, e.g., the example page 223, \citep{key60}). In classical textbooks, points where the Jordan form of $A_0(x)$ is unstable, that is to say either the multiplicity of the eigenvalues or the degrees of the elementary divisors are not constant in a neighborhood of such points, are referred to as \textit{turning points} (see, e.g., page 57, \citep{key59}). The behavior of solutions of differential systems around such points is far more complicated than that of system \eqref{particular} around an \textit{irregular singularity}. In fact, the neighborhood of such a point is decomposed into a finite number of $x$-dependent neighborhoods in each of which the solution behaves quite differently though it may still be asymptotically accomplished (see, e.g., \citep{key62}). In particular, if $h \leq -1$ then the solution of system \eqref{general} can be sought upon presenting the solution as a power series in $\epsilon$. The latter can then be inserted into \eqref{general} and the like powers of $\epsilon$ equated. This reduces the problem  to solving recursively a set of non-homogeneous linear singular differential systems over $K_x$, the first of which is system \eqref{particular} (see, e.g. page 52, \citep{key61} and page 101, \citep{key59}). Hence, it seems plausible to investigate wether $h$ rather than $p$ can be diminished for such systems. 

The methods proposed in the literature of system \eqref{general} either exclude \textit{turning point} \citep{key61} or are not algorithmic throughout. Moreover, they make an essential use of the so-called Arnold-Wasow form.  For the univariate case, system \eqref{particular}, the research advanced profoundly in the last two decades making use of methods of modern algebra and topology. The former classical approach is substituted by efficient algorithms (see, e.g., \citep{key24,key27,key26,key33} and references therein). It was the hope of Wasow \citep{key60}, in his 1985 treatise summing up contemporary research directions and results on system \eqref{general}, that techniques of system \eqref{particular} be generalized to tackle the problems of system \eqref{general}. 

In this article, we are interested in recovering system \eqref{general} from its \textit{turning points} and decoupling it into a set of systems of lower dimensions for each of which $h$ has a minimal integer value.

Given system \eqref{particular}, Moser defined two rational numbers:
\begin{eqnarray}\label{invariant1} m(A) & =& \; \text{max} \;  (0, p + \frac{rank(A_0)}{n}) \quad \text{and} \\  \label{invariant2} \mu(A) &=& \; \text{min} \; \{ \; m (T^{-1} \Delta T)\; | \; T \in GL(V)  \} .
\end{eqnarray}
It follows that system \eqref{particular} is regular whenever $\mu(A) \leq 1$. For $m(A) >1$, he proved that $m(A) > \mu(A)$ if and only if the polynomial 
$$\theta(\lambda) := {x}^{rank (A_0)} \; det(\lambda I + \frac{A_0}{x} + A_1 )|_{x=0}$$ 
vanishes identically in $\lambda$. In this case, system \eqref{particular} (resp. $A(x)$) is said to be Moser-reducible and $m(A)$ can be diminished by applying a coordinate transformation $Y= T Z$ where $T \in GL(V)$ of the form
$$T= (P_0 + P_1 x)\; diag (1, \dots, 1, x, \dots, x)$$
where $P_0 , P_1$ are constant matrices and $det (P_0) \neq 0$ [Theorems 1 and 2, page 381, \citep{key19}]. 
This notion and algorithms developed (see, e.g. \citep{key26}), were generalized as well to linear functional matrix equations in \citep{key64},  a particular case of which is system \eqref{particular}. This gave rise to the package ISOLDE  \citep{key27}, written in the computer algebra system Maple and dedicated to their symbolic resolution. Moser's reduction criterion was also borrowed from the theory of differential systems to investigate efficient algorithmic resolution of the perturbed algebraic eigenvalue-eigenvector problem in \citep{key54}. This problem is another prominent example in the univariate case, for which $\delta$ is the zero map and $K=\mathbb{C}((\epsilon))$. Thereby, $\Delta$ is just a linear operator in the standard way and $A(\epsilon)$ is the widely studied perturbation of the constant matrix $A_0$ (see, e.g., \citep{key56}). 

However, despite their utility and efficiency in these univariate cases, Moser-based algorithms are not considered yet over bivariate fields. In particular, the second author of this article developed in  \citep{key41} a Moser-based algorithm for the differential systems associated to $\delta= x \frac{d}{dx}$ and $K= \mathbb{Q}(x)$, the field of rational functions in $x$. This algorithm is the one we generalize here to system \eqref{general}. Although Moser's motivation for introducing these rational numbers for system \eqref{particular} was to distinguish regular singular systems in the sense of Poincar\'e, we show hereby that they serve as well for the reduction of $h$ to its minimal integer value. 
 
This article is organized as follows: In Section \ref{parampreliminaries}, we give preliminaries. In Section \ref{nested}, we present a Moser-based method of recovery from \textit{turning points}.  In Section \ref{moser}, we introduce a Moser-based reduction algorithm to diminish $h$ to its minimal integer value. Our main results are Theorem \ref{mosernilpotent} and Proposition \ref{nestedprop}. In Subection \ref{compare}, we give an outline of a generalization of Levelt's rank reduction algorithm of \citep{key57} and illustrate by an example its comparison to the Moser-based one. Finally, we conclude in Section \ref{conclusion} and point out some prospects of further research. We remark however that as our interest is formal, reference to asymptotic behavior and/or convergence is dropped. One may consult in the latter direction, e.g., \citep{key6, key60}.

\section{Preliminaries}
\label{parampreliminaries}
Consider system \eqref{general}. Without loss of generality, we can assume that $A_k (x) \in {K}^{n \times n}_x$ and the leading coefficient matrix $A_0(x)$ is a nonzero element of $\mathcal{O}_x^{n \times n}$ otherwise $h$ and $p$  can be readjusted. We refer to $A_{00}:= A(0, 0)$ as the leading constant matrix.  The nonnegative integer $h$ will be referred to as the \textit{rank} of singularity in $\epsilon$ ($\epsilon$-rank, in short).  We assume that system \eqref{general} has at most one \textit{turning point} otherwise the region of study can be shrinked. Moreover, this turning point is placed at the origin, otherwise, a translation in the independent variable can be performed. We denote by $I_k$ and $O_{k,l}$ the identity and zero matrices respectively of prescribed dimensions. 

Let $T \in GL_n(K)$ then the change of the dependent variable $Y = T Z$ in \eqref{general} gives 
\begin{equation} \label{gauge} \epsilon  \frac{dZ}{dx} = \tilde{A}(x, \epsilon) Z= \epsilon^{- \tilde{h}} x^{- \tilde{p}} \sum_{k=0}^{\infty} A_k(x) {\epsilon}^k  Z. \end{equation}
where $\tilde{A}(x, \epsilon) \in K^{n \times n}$ and 
\begin{equation} \label{gaugerelation} T[A]:= \tilde{A}(x, \epsilon)= T^{-1} A(x, \epsilon) T - \epsilon \;T^{-1} \frac{d}{dx}{T} .\end{equation}
Systems \eqref{general} and \eqref{gauge} are called \textit{equivalent}. In fact, $T$ is a change of basis and $\tilde{A}(x, \epsilon)$  is the matrix of $T^{-1} \Delta T$. 
\begin{remark}
\label{noeffect1} Given system \eqref{general} and an \textit{equivalent} system \eqref{gauge}. 
\begin{itemize}
\item If $h>0$ and $T(x) \in GL_n(K_x)$ then  it follows from \eqref{gaugerelation}
%$x^{-\tilde{p}} (\tilde{A}_0(x) + \tilde{A}_1(x) \epsilon + \dots)  &=& T^{-1} x^{-\tilde{p}} (\tilde{A}_0(x) + \tilde{A}_1(x) \epsilon + \dots) T - x \epsilon^{h+1} T^{-1} \frac{d T}{dx}.$ So, we have
$$x^{-\tilde{p}} (\tilde{A}_0(x) + \tilde{A}_1(x) \epsilon) = T^{-1} x^{-p} (A_0(x) + A_1(x) \epsilon) T .$$
Hence, since we'll be interested solely  in $\tilde{A}_0(x)$ and/or $\tilde{A}_1(x)$, it suffices to investigate $T^{-1} A T$, the similarity term of $T[A]$.  In particular, if $T(x) \in GL_n(\mathcal{O}_x)$ then $p=\tilde{p}$ as well. 
\item If $T= diag (\epsilon^{\alpha_1}, \dots, \epsilon^{\alpha_n})$, where $\alpha_1, \dots, \alpha_n$ are integers, then $\delta{T}$ is a zero matrix. Hence, $T[A]= T^{-1} A T$. 
\end{itemize}
\end{remark} 
A classical tool in perturbation theory is the so-called splitting which separates off the existing distinct coalescence patterns. Whenever the leading constant matrix $A_{00}$ admits at least two distinct eigenvalues, there exists a coordinate transformation $T \in GL_n(\mathcal{O})$ which block-diagonalizes System \eqref{general} so that it can be decoupled into subsystems of lower dimensions whose leading constant matrices have a unique distinct eigenvalue each (see, e.g., Theorem $XII-4-1$ page 381, \citep{key71}).
\begin{remark} In (Theorem $2.3-1$, page $17$, \citep{key60}), we have a generalization of the splitting to the well-behaved case which refers to $A_0(x)$ being holomorphically similar to its holomorphic Jordan form. In particular, its eigenvalues don't coalesce in the region of study. However, we limit ourselves to the weaker version given above as we aim to give a general discussion which doesn't exclude \textit{turning points}.  
\end{remark}
We now consider one of these resulting subsystems and assume its leading constant matrix is in jordan normal  form with a unique eigenvalue $\gamma \in \mathbb{C}$. This can be always attained by a constant transformation.  Upon applying the so-called eigenvalue shifting, i.e. $$Y = exp (\epsilon^{-h-1} \int \gamma x^{-p} dx) Z, $$ it is easy to verify from \eqref{gaugerelation} that the resulting system has a nilpotent leading constant matrix. Hence, without loss of generality, we can assume that system \eqref{general} is such that $A_{00}$ is nilpotent. Clearly, it doesn't follow that $A_0(x)$ is nilpotent and we deviate here from the classical treatment of system \eqref{particular} as we may encounter \textit{turning points}. 

The subscripts are to be ommitted and $x, \epsilon$ dropped from the notation whenever ambiguity is not likely to arise, e.g. $A(x, \epsilon)$ , $A_0(x)$, $A_1(x)$ will be denoted by $A, A_0, A_1$ respectively. We set $r= rank(A_0(x))$. 
\section{Recovery from Turning Points}
\label{nested}
We arrive at this section with $A_0(x)$ of system \eqref{general} being  nonnilpotent in contrary to the leading constant matrix $A_{00}$. We show that by Moser-based reduction of $A_0(x)$ itself and ramification in $x$, that is a readjustment $x=t^s$ with $s$ a positive integer, we can modify radically the nilpotency of $A_{00}$, i.e. arrive at a system for which splitting can be applied.  

The motivation behind considering such a general form of system \eqref{general} rather than that for which $p \leq 0$ is to be justified in this section. In fact, system \eqref{general} will undergo a sequence of transformations which might introduce or elevate the order of the pole in $x$. This elevation is inevitable and is introduced identically in the classical treatment of such systems. However, the order of poles of $A_k(x), k=0, 1, 2, \dots$ grows at worst linearly with $k$ which maintains an asymptotic validity of the formal construction [page 60, \citep{key59}].

By Remark \eqref{noeffect1}, the discussion is restricted to the similarity term $T^{-1} A T$ of $T[A]$. Hence, $A_0(x)$ can be treated as a perturbation of $A_{00}$. 
\begin{proposition} 
\label{nestedprop}
Given system \eqref{general}
$$
\epsilon \frac{dY}{dx} = A(x, \epsilon) Y = \epsilon^{- h} x^{- p} \sum_{k=0}^{\infty} A_k(x) {\epsilon}^k  Y
$$
where $A_0(x) = A_{00} + A_{01} x + A_{02} x^2 + \dots$ such that $A_{00}$ is a nilpotent constant matrix. If $A_0(x)$ has a nonzero eigenvalue then there exist a positive integer $s$ and a $T \in GL_n(K_x)$ such that by setting $x=t^s$, $Y = T Z$, we have in \eqref{gauge}
$$\tilde{A}_0(t) := T^{-1} A_0(t) T= \tilde{A}_{00} + \tilde{A}_{01} t + \tilde{A}_{02} t^2 + \dots$$
where  $\tilde{A}_{00}$ has a nonzero constant eigenvalue. 
\end{proposition}
\begin{proof}
The eigenvalues of $A_0(x)$ admit a formal expansion in the fractional powers of $x$ in the neighborhood of $x=0$ (see, e.g. \citep{key56}).  We are interested only in their first nonzero terms. Let $\mu(x) = \sum_{j=0}^{\infty} \mu_j x^{j/s}$ be a nonzero eigenvalue of $A_0(x)$ whose leading exponent, i.e. smallest $j/s$ for which  $\mu_j \neq 0$ and $j, s$ are coprime, is minimal among the other nonzero eigenvalues. Then, without loss of generality, we can assume that $s=1$ otherwise we set $x = t^{s}$. Now, let $T \in GL_n(K_x)$ such that $\tilde{A}_0(x)$ is Moser-irreducible in $x$ i.e. 
$$\theta(\lambda) = {x}^{rank(\tilde{A}_{00})} \; det(\lambda I + \frac{\tilde{A}_{00}}{x} + \tilde{A}_{01} )|_{x=0}$$
doesn't vanish identically in $\lambda$. Then there are $n - deg \theta$ eigenvalues of $A_0(x)$ whose leading exponents lie in $[0,1[$, and deg $\theta$ eigenvalues for which the leading exponent is greater or equal 1 \citep 1in \citep{key54}]. But $\mu(x)$ is an eigenvalue of $\tilde{A}_0(x)$ with a minimal leading exponent and hence it is among those whose leading exponents lie in $[0,1[$. By our assumption, $s=1$ and hence the leading exponent of $\mu(x)$ is zero and $\mu_0 \neq 0$. Since $\mu_0$ is an eigenvalue of $\tilde{A}_{00}$, it follows that the latter is nonnilpotent.
\end{proof}
\begin{remark}
The eigenvalues of $A_0(x)$ are the roots of the algebraic scalar equation $f(x, \mu)= det (A(x) - \mu I_n) = 0$ and can be computed by Newton-Puiseux algorithm. The linear transformation $T \in GL_n(K_x)$ can be computed efficiently via ISOLDE \citep{key27}.
\end{remark}
We illustrate our approach by an Example from [page 88, \citep{key59}]. We remark however that the technique proposed in the former, adapted from \citep{key70} and \citep{key62}, debutes with the reduction of $A(x, \epsilon)$ to its Arnold-Wasow  form as mentioned above and then constructing the so-called characteristic polygon. This particular example was given initially in that for simplicity and hence, the transformations computed by our algorithm, which doesn't require reduction to this form, do not deviate hereby from those in the former. 
\begin{example}
Let $\epsilon  \frac{dY}{dx} = \epsilon^{-1}  \begin{bmatrix}  0 & 1& 0 \\ 0 & 0 & 1 \\  \epsilon & 0 & x \end{bmatrix} Y $. Clearly, we have a truning point at $x=0$ and
$$A_0(x)= \begin{bmatrix}  0 & 1& 0 \\ 0 & 0 & 1 \\0 & 0 & x \end{bmatrix} .$$ $A_{00}$ is nilpotent in Jordan form while the eigenvalues of  $A_0(x)$ are $0, 0, x$ whence $s=1$. A simple calculation shows that $A_0(x)$ is Moser-reducible in $x$ ($\theta(\lambda) \equiv 0$). Let $T= diag(1,x,x^2)$ then upon setting $Y=TZ$ we get
\begin{eqnarray*} \epsilon \frac{dZ}{dx} = \epsilon^{-1}  x \{ \begin{bmatrix}  0 & 1& 0 \\ 0 & 0 & 1 \\ 0& 0 & 1\end{bmatrix}  +  \begin{bmatrix}  0 & 0& 0 \\ 0 & 0 & 0 \\ 1& 0 & 0\end{bmatrix} x^{-3} \epsilon + \\ \begin{bmatrix}  0 & 0& 0 \\ 0 & -1 & 0 \\ 0& 0 & -2 \end{bmatrix} x^{-2} {\epsilon}^{2} \} Z .\end{eqnarray*}
The constant leading coefficient matrix of the resulting matrix is no longer  nilpotent. Hence the system can be decoupled into two subsystems upon setting $Z= T W$ where $$T=\begin{bmatrix}  1 & 0& 1 \\ 0 & 1 & 1 \\ 0& 0 & 1\end{bmatrix} + \begin{bmatrix}  -1& -1&0 \\ -1 & -1& -2 \\ -1& -1 & 0\end{bmatrix} x^{-3} \epsilon + O(x^{-6}\epsilon^{2}). $$  The resulting equivalent system then consists of the two uncoupled lower dimension systems where $W= {[W_1, W_2]}^T$:
\begin{eqnarray*} \epsilon \frac{dW_1}{dx} =  \epsilon^{-1} x \{ \begin{bmatrix}  0 & 1\\ 0 & 0  \end{bmatrix}  +  \begin{bmatrix}  -1 & 0 \\ -1 & 0 \end{bmatrix} x^{-3} \epsilon + \\ \begin{bmatrix}  1 & -1 \\ 1 & -1 + x^4  \end{bmatrix} x^{-6} {\epsilon}^{2} &+& O(x^{-9} \epsilon^3)\}  W_1 . \end{eqnarray*} 
$$\epsilon  \frac{dW_2}{dx} =  \epsilon^{-1} x \{ 1+ x^{-3} \epsilon + (1+ 2x^4) x^{-6} {\epsilon}^{2} + O(x^{-9} \epsilon^3)\} W_2 .  $$
For the former subsystem, $A_0(x)$ and $A_{00}$ are now simultaneously nilpotent. As for the latter,  the exponential part is $\frac{1}{2} \epsilon^{-2} x^2 (1+ O(x^{-3} \epsilon))$ which exhibits that the eigenvalue $x$ of the leading coefficient matrix of the system we started with is recovered as expected. 
\end{example}
\section{$\epsilon$-Rank Reduction}
\label{moser}
Following Section \ref{nested}, we can now assume without loss of generality that $A_0(x)$ and $A_{00}$ are simultaneously nilpotent. The system is recovered from its \textit{turning points} and no further reduction can be attained via Splitting. In analogy to modern techniques for system \eqref{particular}, we investigate $\epsilon$-rank reduction of system \eqref{general}.
\subsection{$\epsilon$-Reduction Criteria}
\label{regularity}
Following \eqref{invariant1} and \eqref{invariant2}, we define respectively the $\epsilon$-Moser rank and $\epsilon$-Moser Invariant of system \eqref{general} as the rational numbers:
\begin{eqnarray*} m_{\epsilon}(A) &= &\; \text{max} \;  (0, h + \frac{r}{n}) \quad 
\text{and} \\ \mu_{\epsilon} (A) &=& \; \text{min} \; \{ m_{\epsilon} (T[A]) | T \in GL_n(K) \; \}. \end{eqnarray*} 
\begin{definition}
System \eqref{general} (the matrix $A$ respectively) is called $\epsilon$-reducible in the sense of Moser ($\epsilon$-reducible, in short)  if $m_{\epsilon}(A) > \mu_{\epsilon} (A)$, otherwise it is said to be $\epsilon$-irreducible.
\end{definition}
\begin{remark}
This definition is not to be mixed neither with the usual sense of reduced system in the literature, i.e. the system \eqref{particular} obtained from system \eqref{general} ($\epsilon =0$); nor with the usual sense of Moser-reduced systems as for system \eqref{particular}, i.e. the systems whose $p$ is minimal.
\end{remark}
In particular, if $m_{\epsilon}(A) \leq1$ then $h=0$ and no further reduction is to be sought via this approach.
\subsection{$\epsilon$-Reduction Algorithm}
\label{reductionalgorithm}
We follow the algorithmic description of \citep{key41}. The main result of this section is the following theorem which is to be proved after giving its useful building blocks. 
\begin{theorem}
\label{mosernilpotent}
Given system \eqref{general}
$$
\epsilon \frac{dY}{dx} = A(x, \epsilon) Y = \epsilon^{- h} x^{- p} \sum_{k=0}^{\infty} A_k(x) {\epsilon}^k  Y
$$
such that $rank(A_0(x))=r$ and $m_{\epsilon}(A)>1$.  A necessary and sufficient condition for the system to be $\epsilon$-reducible (in the sense of Moser), i.e. the existence of a $T(x, \epsilon) \in Gl_n(K)$ such that $r (\tilde{A}_0(x)) < r$, is that the polynomial 
$$\theta(\lambda) := {\epsilon}^{r} \; det(\lambda I + \frac{A_0}{\epsilon} + A_1 )|{\epsilon=0}$$ 
vanishes identically in $\lambda$. \\
Moreover, $T(x, \epsilon)$ can always be chosen to be a product of unimodular transformations in $GL_n(\mathcal{O}_x)$ and polynomial transformations of the form $diag (\epsilon^{\alpha_1}, \dots, \epsilon^{\alpha_n})$ where $\alpha_1, \dots, \alpha_n$ are nonnegative integers. 
\end{theorem}
\begin{lemma}
\label{gauss1}
There exists a unimodular matrix $U(x)$ in $GL_n(\mathcal{O}_x)$ such that the leading coefficient matrix $\tilde{A}_0(x)$ of $U[A]$ have the following form
\begin{equation} \label{gaussform} \tilde{ A_0}  = \begin{bmatrix} \tilde{A}_0^{11}  & O \\ \tilde{A}_0^{21} & O \end{bmatrix} \end{equation} 
where  $\tilde{A}_0^{11} $ is a square matrix of dimension $r$ and $\begin{bmatrix} \tilde{A}_0^{11} \\ \tilde{A}_0^{21} \end{bmatrix}$ is a $n \times r$ matrix of full column rank $r$. 
\end{lemma}
\begin{proof}
By Remark \ref{noeffect1}, $\tilde{A}_0 = U^{-1} A_0 U$ hence it suffices to search a similarity transformation $U(x)$. Since $\mathcal{O}_x$ is a Principal Ideal Domain (the ideals of $\mathcal{O}_x$ are of the form $x^k \mathcal{O}$), it is well known that one can construct unimodular transformations $Q(x), R(x)$ lying in $GL_n(\mathcal{O}_x)$ such that the matrix $Q A_0 R$ has the Smith Normal Form
$$Q\;A_0\;R = diag (x^{\beta_1}, \dots, x^{\beta_r}, 0, \dots, 0) )$$
where $det R(0) \ne 0$, $det Q(0) \ne 0$ , and $\beta_1, \dots, \beta_r$ in $\mathbb{Z}$ with $0 \leq \beta_1 \leq \beta_2 \leq \dots \leq \beta_r$.\\
It follows that we can compute a unimodular matrix $R(x)$ in  $GL_n(\mathcal{O}_x)$ so that its $n-r$ last columns form a basis of $ker(A_0)$.   Hence, we set $U(x)=R(x)$. 
\end{proof}
\begin{remark}
\label{gauss2}
In practice, $U(x)$  can be obtained by performing gaussian elimination on the columns of $A_0(x)$ taking as  pivots  the elements of minimum order (valuation) in $x$ as already suggested in \citep{key5}.
\end{remark}
Hence, we can suppose without loss of generality that  $A_0(x)$ consists of $r$ independent columns and $(n-r)$ zero columns. We partition $A_1(x)$ conformally with $A_0$, i.e. $ A_1= \begin{bmatrix} A_1^{11}& A_1^{12} \\ A_1^{21} & A_1^{22}  \end{bmatrix}$, and consider  
\begin{equation} \label{glambdaform} G_{\lambda}(A)= \begin{bmatrix} A_0^{11} & A_1^{12} \\ A_0^{21} & A_1^{22} + \lambda  I_{n-r}\end{bmatrix} .\end{equation}
We illustrate our progress with the following simple example. 
\begin{example}
\label{firststep}
Given $\epsilon \frac{d Y}{dx} = A(x, \epsilon) Y $ where $$A= \epsilon^{-2}  \begin{bmatrix} \epsilon  & -x^3 \epsilon & (1+x) \epsilon & 0 \\ x^2 & x \epsilon & 0 & -2x \epsilon \\ -x & 0 & 0 & 2 \epsilon \\ 0 & 2 & 0 & \epsilon^2\end{bmatrix} . $$ Clearly, $A_0(x)$ is nilpotent of rank $2$ and \begin{equation} \label{firststepg} G_{\lambda}(A)= \begin{bmatrix}  0 & 0 & x+1  & 0 \\ x^2 & 0 & 0 & -2 x\\ -x & 0 &  \lambda  & 2 \\ 0 & 2 & 0 &  \lambda  \end{bmatrix} .\end{equation}
\end{example}
This consideration of $G_{\lambda}(A)$ gives an $\epsilon$-reduction criteria equivalent to $\theta(\lambda)$ as demonstrated in  the following lemma.
\begin{lemma}
\label{glambda}
$Det(G_{\lambda}(A) \equiv 0$ vanishes identically in $\lambda$ if and only if $\; \theta(\lambda)$ does.   
\end{lemma}
\begin{proof}
Let $D(\epsilon)= diag(\epsilon I_{r}, I_{n-r})$ then we can write $\epsilon^{-1} A= N D^{-1}$ where $N:= N(x, \epsilon) \in K^{n \times n}$ has no poles in $\epsilon$. Set $D_0 = D(0)$ and $N_0= N(x,0)$. Then we have 
\begin{eqnarray*} det(G_{\lambda}(A) & =&  det (N_0 + \lambda D_0)  =  det (N + \lambda D)|_{\epsilon=0} \\
& = & (det (\frac{A}{\epsilon} + \lambda I_n) \; det (D))|_{\epsilon=0}  \\ &=&  (det (\frac{A_0}{\epsilon} + A_1 + \lambda I_n) \; \epsilon^{r} )|_{\epsilon=0} = \theta (\lambda).  \end{eqnarray*} 
\end{proof}
\begin{proposition}
\label{gauss3}
If $m_{\epsilon}(A) >1$ and $det(G_{\lambda}(A)) \equiv 0$ then there exists a unimodular matrix $Q(x)$ in $GL_n(\mathcal{O}_x)$ with $det \; Q(x) = \pm 1$ such that the matrix $G_{\lambda}(Q[A])$ has the form 
\begin{equation} \label{particularform3} G_{\lambda}(Q[A]) = \begin{bmatrix} A_0^{11} & U_1 & U_2 \\ V_1 & W_1 + \lambda I_{n- r -\rho} & W_2 \\ M_1 & M_2 & W_3 + \lambda I_\rho \end{bmatrix} ,\end{equation}
where $0 \leq \rho \leq n-r,\; W_1,\; W_3$ are square matrices of order $(n-r-\rho)$ and $\rho$ respectively , and
\begin{equation} \label{particularconditionb} rank(\begin{bmatrix} A_0^{11} & U_1 \\ M_1 & M_2 \end{bmatrix})= rank(\begin{bmatrix} A_0^{11} & U_1 \end{bmatrix}),\end{equation}
\begin{equation}  \label{particularconditiona}rank(\begin{bmatrix} A_0^{11} & U_1 \end{bmatrix}) < r . \end{equation}
\end{proposition}
Our procedure is that of \citep{key41} except for the properties of $M_1$, $M_2$, and $W_3$. The nullity of $M_1$ and $M_2$ in the former is replaced by the weaker condition \eqref{particularconditionb} here. Otherwise, the unimodularity of  $Q(x)$ cannot be guaranteed. Moreover, this refinement avoids unnecessary computations. The bridge between both descriptions is established in Remark \ref{gq} before proceeding to the proof of the Proposition. 
\begin{remark}
\label{gq}
Suppose that $G_{\lambda}(Q[A])$ has the form \eqref{particularform3} and there exists a transformation $T(x) \in GL_n(K_x)$ such that $G_{\lambda}(T[Q[A]])$ has the form
$$G_{\lambda}(T[Q[A]]) = \begin{bmatrix} A_0^{11} & U_1 & U_2 \\ V_1 & W_1 + \lambda I_{n- r -\rho} & W_2 \\ O & O & \tilde{W}_3 + \lambda I_\rho \end{bmatrix} ,$$
where $0 \leq \rho \leq n-r,\; W_1,\; \tilde{W}_3$ are square matrices of order $(n-r-\rho)$ and $\rho$ respectively, and $\tilde{W}_3$ is upper triangular with zero diagonal. Then, $$ det\; G_{\lambda}(T[Q[A]]) = {\lambda}^\rho \; det \;  \begin{bmatrix} A_0^{11} & U_1 \\ V_1 & W_1 + \lambda I_{n- r -\rho} \end{bmatrix}.$$
If $det(G_{\lambda}(A) \equiv 0)$ then we have as well
 $det(G_{\lambda}(Q[A])= det(G_{\lambda}(T[Q[A]]) \equiv 0 . $ Hence,
\begin{equation}
\label{singular}
det \;  \begin{bmatrix} A_0^{11} & U_1  \\ V_1 & W_1 + \lambda I_{n- r -\rho} \end{bmatrix} = 0 .
\end{equation}
We shall denote by $G_0^{(\rho)}$ the matrix $G_0^{(\rho)} = \begin{bmatrix} A_0^{11} & U_1  \\ V_1 & W_1 \end{bmatrix}. $
\end{remark}
\begin{proof} (Proposition \ref{gauss3})  Since $det(G_{\lambda}(A))  \equiv 0$ then in particular, $det(G_{\lambda=0}(A)) = 0$. The matrix $G_{\lambda=0}(A)$ is thus singular. Let $E$ (respectively $F$) be the vector space spanned by the first $r$ (resp. last $n-r$) rows of $G_{\lambda=0}(A)$. We have
$$dim (E+F) = rank(G_{\lambda=0}(A)) < n .$$
If  $dim (E) < r$ then set $\rho =0$. Otherwise, since  $$dim(E + F)  = dim(E) + dim(F) - dim(E \cap F) < n , $$   it follows that 
%$dim(F) - dim(E \cap F)< n-r$. We see
 either $dim(F) < n-r$ or $dim(E \cap F) > 0$. In both cases, there exists at least a  row vector $W_1(x)=(w_1, \dots, w_n)$ with entries in $\mathbb{C}((x))$ in the left null space of $G_{\lambda =0}(A)$, such that $w_i \neq 0$ for some $r+1 \leq i \leq n$. We can assume without loss of generality that $W_1(x)$ has its entries in $\mathcal{O}_x$. Indeed, this assumption can be guaranteed by a construction as in Lemma \ref{gauss1}. Let  $V_1(x)$ be the nonzero row vector whose entries $v_i$ are such that $v_i = w_{i}$ for  $r+1 \leq i \leq n$ and zero elsewhere. It can be attained that $v_n=1$ upon replacing $A$ by $P[A]$ where $P$ is a permutation matrix which exchanges the rows of $A$ so that the pivot is taken to have the minimum order (valuation) in $x$. 
 Let $R_1(x)$ be the matrix whose $n^{th}$ row is formed by $V_1(x)$, the other rows being those of the identity matrix $I_n$. Note that $R_1(x)$ is a lower triangular matrix in $GL_n(\mathcal{O}_x)$ with $det\; R_1 =1$. Now put $A^{(1)}= R_1^{-1}[A]$.  Thus, $G_{\lambda}(A^{(1)})$ has the form \eqref{particularform3} with \eqref{particularconditionb} and $\rho=1$. 
\\ By \eqref{singular}, the matrix $G_0^{(1)}(x)$ is singular so if the condition \eqref{particularconditiona} does not occur then one can find, by same argument as above, a vector $V_2(x)$ in $\mathcal{O}_x$ whose $(n-1)^{th}$ component is $1$ such that $V_2 \; G_0^{(1)} =0$.  If $R_2(x)$ denotes the matrix whose ${n-1}^{th}$ row is formed by $(V_2(x), 0)$, the others being those of the identity matrix $I_n$, and $A^{(2)} = R_2^{-1}[A^{(1)}]$ then the matrix $G_{\lambda}(A^{(2)})$ has the form \eqref{particularform3} with \eqref{particularconditionb} and $\rho=2$. 

A finite number of applications of this process yields an equivalent matrix $A^{(\rho)}$ with \eqref{particularconditionb} for which either \eqref{particularconditiona} occurs or $\rho=n-r$. But in the latter case one has, again by Remark \ref{gq} and \eqref{singular}, that $det\; A_0^{11} =0$, and so \eqref{particularconditiona} occurs. The matrix $Q(x)$ we are seeking is obtained as a product of permutation matrices or lower triangular matrices whose determinant is $1$, and so its determinant is $\pm 1$. 
\end{proof}

We consider again Example $\ref{firststep}$.
\begin{example*}
A simple calculation shows that $det(G_{\lambda}(A)) \equiv 0$ hence $A$ is $\epsilon$-reducible.  From \eqref{firststepg}, for $\lambda =0$, we have the singular matrix
$$G_{\lambda=0}(A) =   \begin{bmatrix}  0 & 0 & x+1  & 0 \\ x^2 & 0 & 0 & -2 x\\ -x & 0 &  0  & 2 \\ 0 & 2 & 0 & 0 \end{bmatrix}.\; \text{Let} \quad Q= \begin{bmatrix}  1 & 0 & 0  & 0 \\ 0 & 1 & 0 & 0\\ 0 & 0 & 0 & 1 \\ 0 & 0 & 1 &  0\end{bmatrix},$$ 
$$ \text{hence} \quad Q[A]=  \epsilon^{-2}  \begin{bmatrix} \epsilon  & -x^3 \epsilon & 0 & (1+x) \epsilon \\ x^2 & x \epsilon & -2x \epsilon & 0 \\ 0 & 2 & \epsilon^2 & 0 \\ -x & 0 & 2 \epsilon & 0 \end{bmatrix} . $$ 
Moreover,  $G_{\lambda}(Q[A])$ has the form \eqref{particularform3} with $\rho= 1$ and $r=2$. In fact,  $G_{\lambda}(Q[A]) = \begin{bmatrix} 0  & 0 & 0 & (1+x)  \\ x^2 & 0 & 0 & 0 \\ 0 & 2 & \lambda & 0 \\ -x & 0 & 2 &  \lambda \end{bmatrix} . $
\end{example*}
\begin{proposition}
\label{shearing}
If $m_{\epsilon}(A) >1$ and $G_{\lambda}(A)$ has the form \eqref{particularform3} with conditions \eqref{particularconditionb} and \eqref{particularconditiona} satisfied, then system \eqref{general} (resp. $A$) is $\epsilon$-reducible and reduction can be carried out with the so-called shearing  transformations $S = diag(\epsilon I_r , I_{n-r-\rho}, \epsilon I_\rho)$ whenever $\rho \neq 0$ and $S=diag(\epsilon I_r , I_{n-r})$ otherwise.
\end{proposition}
\begin{proof}
We partition $A(x, \epsilon)$ as follows
$$A= \begin{bmatrix} A^{11}  & A^{12} & A^{13} \\ A^{21}  & A^{22} & A^{23} \\  A^{31}  & A^{32} & A^{33} \end{bmatrix}$$
where $A^{11} , A^{22} , A^{33} $ are of dimensions $r, n-r-\rho,$ and  $\rho$ respectively. 
It is easy to verify then  that $$ \tilde{A}(x, \epsilon)= S^{-1} A S - S^{-1} \delta S= S^{-1} A S= \begin{bmatrix} A^{11}  & \epsilon^{-1} A^{12} & A^{13} \\   \epsilon A^{21}  & A^{22} &  \epsilon A^{23} \\  A^{31}  &  \epsilon^{-1} A^{32} & A^{33} \end{bmatrix}.$$
Hence, the new leading coefficient matrix is 
$$\tilde{A}_0 = \begin{bmatrix} A_0^{11} & U_1 & O \\ O& O &O \\ M_1 & M_2 & O \end{bmatrix}$$  and $rank (\tilde{A}_0) = rank( A_0^{11} \quad  U_1) < r.$ \end{proof}

\begin{example*}
Let $S=diag(\epsilon , \epsilon, 1, \epsilon)$ then 
$$S[Q[A]] = \epsilon^{-2} \begin{bmatrix} \epsilon & -x^3 \epsilon  & 0 & (1+x) \epsilon \\ x^2 &x \epsilon & -2x & 0 \\ 0 & 2  \epsilon & \epsilon^2 & 0 \\ -x & 0  & 2 & 0\end{bmatrix} . $$ It is clear that the leading term has rank $1< r$. 
\end{example*}
Now, we are ready to give the proof of  Theorem \ref{mosernilpotent}.

\begin{proof} (Theorem \ref{mosernilpotent}) We first prove that the condition is necessary. In fact, suppose that such a $T$ exists then we have
$det (\lambda I + \frac{A}{\epsilon}) = det (\lambda I + \frac{\tilde{A}}{\epsilon}).$
It then suffices to see that
$$ \epsilon^r \; det (\lambda I + \frac{A}{\epsilon})|_{\epsilon=0} =  \epsilon^{r} \; det (\lambda I + \frac{\tilde{A}}{\epsilon})|_{\epsilon=0} =  \epsilon^{r-m} \; P(x, \epsilon)|_{\epsilon=0}$$
where $m \leq r(\tilde{A}_0) < r \;\text{and}\; P(x, \epsilon) \in K^{n \times n}$ has no poles in $\epsilon$. Moreover, it is evident that $\theta (\lambda)=0$ since   $$ \epsilon^r \; det (\lambda I + \frac{A}{\epsilon})|_{\epsilon=0} = \epsilon^r \; det (\lambda I + \frac{A_0}{\epsilon} + A_1 + \sum_{k=2}^{\infty} A_k {\epsilon}^{k-1})|_{\epsilon=0}.$$
As for the sufficiency of the theorem, by Lemma \ref{gauss1}, we can assume that $A_0(x)$ is in the form \eqref{gaussform}. Then,  $G_{\lambda}(A)$ is constructed as in \eqref{glambdaform}. Since $\theta(\lambda)=0$ and $m_{\epsilon}(A) >1$, it follows from Lemma \ref{glambda} that $det(G_{\lambda}(A))=0$ and $A$ is $\epsilon$-reducible. Hence, the matrix $S[Q[A]]$ where $S, Q$ are as in Propositions \ref{gauss3} and \ref{shearing} respectively, has the desired property. 
\end{proof}
\begin{algorithm}
\caption{Moser-based $\epsilon$-Rank Reduction of System \eqref{general}}
\label{algorithmparam}
\textbf{Input:} $A(x,\epsilon)$ of system \eqref{general}  \\
\textbf{Output:} $T(x, \epsilon) \in GL_n(K)$ and an \textit{equivalent} system given by $T[A]$ which is $\epsilon$-irreducible in the sense of Moser.\\
\begin{algorithmic}
\State $T  \gets  I_{n}$; \\
$h  \gets  \epsilon$-rank of $A$;\\
$U(x)  \gets $ unimodular transformation  computed in Lemma \ref{gauss1} and Remark \ref{gauss2} such that $U^{-1} A_0 U$ has form \eqref{gaussform};\\
$T  \gets  T U$;
$A \gets  U^{-1} A U - \epsilon U^{-1} \frac{d U}{dx}$;
\While {$Det (G_{\lambda}(A))=0$ and $h>0$} \do \\
\State $Q(x), \rho  \gets $ Proposition \ref{gauss3} ;
\State $S(\epsilon)  \gets $ Proposition \ref{shearing};
\State $P  \gets  Q S$;
\State $T  \gets  T P$;
\State $A \gets  P^{-1} A P - \epsilon P^{-1} \frac{d P}{dx}$;
\State $h  \gets  \epsilon$-rank of $A$;
\State $U(x)  \gets $  Lemma \ref{gauss1} and Remark \ref{gauss2};
\State $T  \gets  T U$;
\State $A \gets  U^{-1} A U - \epsilon U^{-1} \frac{d U}{dx}$;
\EndWhile . \\
\Return{(T, A)}.
\end{algorithmic}
\end{algorithm}
\begin{remark} The $\epsilon$-reducibility of $A$ implies that the rank of the leading coefficient matrix can be reduced without necessarily reducing the $\epsilon$-rank of the system. If the $\epsilon$-reduction criteria holds for a sufficient number of equivalent systems then a repetitive application of such a transformation results in an equivalent system whose leading coefficient matrix has a null rank, hence $h$ is reduced by one (e.g. Example \ref{algoexm}). At this point, the discussion restarts from the nature of the eigenvalues of the leading constant matrix. 
\end{remark}
\begin{example}
\label{algoexm}
Given $\epsilon \frac{d Y}{dx} = A(x, \epsilon) Y $ where $$A= \epsilon^{-3}  \begin{bmatrix} 2 x \epsilon^3  & 3 x^2 \epsilon^4 & 2x \epsilon^2 & (2x+1) \epsilon^5 \\ 0 &  \epsilon^4 & 0 & 0\\ 0 & 0 & \epsilon^2 & 0 \\ -2 x & 0 & 0 & 0 \end{bmatrix} . $$
The gauge transformaiton $Y=T Z$ computed by our algorithm results in an equivalent $\epsilon$-irreducible system whose $\epsilon$-rank is diminished by two as illustrated below:
 $$T= \begin{bmatrix} 0& \epsilon^3 & 0 & 0 \\ 0 &  0 & \epsilon & 0\\ \epsilon^3  & 0 &0 & 0 \\ 0 & 0 & 0 & 1 \end{bmatrix}  \; \text{and} \;  T[A]= \epsilon^{-1}  \begin{bmatrix} 1  & 0& 0& 0 \\ 2x &  2 x \epsilon & 3 x^2 & 2x+1\\ 0 & 0 & \epsilon^2 & 0 \\ 0 & -2 x \epsilon & 0 & 0\end{bmatrix} . $$
\end{example}
\subsection{Example of Comparison with Levelt's Approach}
\label{compare}
While Moser defined two rational numbers in \citep{key19}, Levelt investigated the existence of stationary sequences of free lattices in \citep{key57}. 
Since Moser-based and Levelt's algorithms serve the same utility, i.e. rank reduction of system \eqref{particular}, it is natural to question their comparison.  The cost analysis in the univariate case of the Moser-based algorithm described in \citep{key41} and that of Levelt's \citep{key57} was given in [page 108, \citep{key73}]. Both algorithms turn out to have an identical cost which suggests a further experimental study so that they can be well-compared. The latter was generalized in \citep{key5} to a certain class of Pfaffian systems over bivariate fields. Consequently, following arguments similar to those of \citep{key5} whenever the leading coefficient matrix is nilpotent and to Section \ref{nested} of this article whenever it's not, Levelt's algorithm seems adaptable to system \eqref{general} as well. Such an attempt would give rise to Algorithm \ref{algorithmlevelt}.

\begin{algorithm}
\caption{Generalization of Levelt's Rank Reduction to System \eqref{general}}
\label{algorithmlevelt}
\textbf{Input:} $A(x,\epsilon)$ of system \eqref{general};  \\
\textbf{Output:} $T(x, \epsilon) \in GL_n(K)$ computed via Levelt's approach and an \textit{equivalent} system given by $T[A]$ which is $\epsilon$-irreducible in the sense of Moser.\\

\begin{algorithmic}
\State $T\gets I_{n}$\\
$i \gets 0$;\\
$h  \gets  \epsilon$-rank of $A$;
\While {$i < n-1$ and $h>0$} \do
\State $r \gets rank (A_0)$;
\State $U(x) \gets$ unimodular transformation of Lemma \ref{gauss1} and Remark \ref{gauss2} such that $U^{-1} A_0 U$ has the form \eqref{gaussform};
\State $S(x) \gets diag(\epsilon I_r, I_{n-r})$;
\State $P \gets U S$;
\State $T \gets T P$;
\State $A \gets  P^{-1} A P - \epsilon P^{-1} \frac{\partial P}{\partial x}$;
\State $\tilde{h} \gets  \epsilon$-rank of $A$;
\If {$\tilde{h} < h$} \State {$i \gets 0$;} \Else \State $i \gets i+1$; \EndIf
\State $h \gets  \tilde{h}$;
\EndWhile . \\
\Return{(T, A)}.
\end{algorithmic}
\end{algorithm}

It is clear that Algorithm \ref{algorithmlevelt} coincides with Algorithm \ref{algorithmparam} for $\rho=0$. For $\rho > 0$, which frequently occurs for matrices of dimension greater than $3$, we ran simple examples of systems \eqref{general}. Despite the identical cost of both algorithms, these examples exhibited that Algorithm \ref{algorithmlevelt} complicates dramatically the coefficient matrices of the system under reduction. One factor in this complication stems from the weak termination criterion of this algorithm. However, even upon adjoining Moser's termination criterion (given by $\theta(\lambda)$) to this algorithm, as suggested in [Secion 5 of \citep{key5}], the result remains less satisfying than that of Algorithm \ref{algorithmparam}. We exhibit here a selected example over $\mathcal{O}=\mathbb{C}[[x, \epsilon]]$. Additional examples along with our implementation are available online at \citep{key100}.
\begin{example}
Let $\epsilon  \frac{dY}{dx} = \epsilon^{-h} x^{-p} A(x,\epsilon) Y $ where $$A= \begin{bmatrix} 0&0&0&\epsilon x (\epsilon x+3)&\epsilon^3 (x+9)&\epsilon x^2 \\ x&0&0&0&9 \epsilon^2 x^2&0\\ x^2&-1&0&\epsilon x (\epsilon x^6+1)&0&0\\x^3+x&-x&0&\epsilon x^2&0&0\\ 0&0&x&-3 \epsilon&0&0\\ x^3-1&5&0&0&0&0\end{bmatrix}.$$
It is easily verified that $A$ is Moser-reducible in $\epsilon$. Furthermore, we have $m_{\epsilon}(A)= h+ \frac{3}{6}$ and $\mu_{\epsilon}(A) = h + \frac{2}{6}$. Hence, it suffices to run only one reduction step, for which the rank of the leading coefficient matrix is dropped by one. We give in the following the transformation $T(x,\epsilon)$ and $T[A]$ as computed by our Moser-based algorithm, the generalization of Levelt's algorithm with Moser's reducibility criterion, and Levelt's algorithm respectively. For the latter, $T$ and $T[A]$ are to be found at \citep{key100} due to the lack of space here. However we illustrate the dramatic growth of their coefficients by listing one entry of each.  
\begin{itemize}
\item $T= \begin{bmatrix} \epsilon&0&0&0&0&0\\ 0&\epsilon&0&0&0&0 \\ 0&0&\epsilon&0&0&0 \\ 0&0&0&0&0&\epsilon \\ 0&0&0&1&0&0 \\ 0&0&0&0&\epsilon&0 \end{bmatrix} \quad \text{and} $
$$T[A]= \begin{bmatrix} 0&0&0&\epsilon^2 (x+9)&x^2 \epsilon&\epsilon x (\epsilon x+3) \\ x&0&0&9 x^2 \epsilon&0&0 \\ x^2&-1&0&0&0&\epsilon x (\epsilon x^6+1) \\ 0&0&\epsilon x&0&0&-3 \epsilon^2 \\ x^3-1&5&0&0&0&0 \\ x (x^2+1)&-x&0&0&0&x^2 \epsilon \end{bmatrix}$$
\item $T= \begin{bmatrix} \epsilon^3&0&0&0&0&-9 x \epsilon^2 \\ 0&0&-1/3 x^2 \epsilon^2&\epsilon^2&0&0 \\ 0&0&0&0&\epsilon&0 \\ 0&\epsilon^2&-1/3 \epsilon x&0&0&0 \\ 0&0&0&0&0&1 \\ 0&0&\epsilon&0&0&0 \end{bmatrix} \quad \text{and} $ \\ $ T[A]= [\tilde{a}_{ij}]$ including entries having $2$-digit coefficients and degree $8$ in $x$, e.g.,$$\begin{cases}
\tilde{a}_{12}=-x (27 \epsilon^2-\epsilon x-3)\\
\tilde{a}_{13}= \frac{1}{3} x^2 (-x+27 \epsilon)\\
\tilde{a}_{41}= \epsilon x^4-\epsilon x+3\\
\tilde{a}_{52}=\epsilon^2 x (\epsilon x^6+1)\\
\tilde{a}_{53}=-(1/3) \epsilon^2 x^8\end{cases}$$
\item $T=  [t_{ij}]$, $T[A]=[\tilde{a}_{ij}]$ have entries with $4$-digit and $10$-digit coefficients. Degrees in $x$ surpass $8$, e.g.
$$\begin{cases}
t_{23}= 1458\epsilon^2 x^5 {(x^3+3x^2+2)}^2/(243x^{11}+1944x^{10}\\+3627x^9+1026x^8+1944x^7-8820x^6+
432x^5\\-4860x^4-72x^3-652x^2-72x-324)
\end{cases}$$
$$
 \begin{cases} \tilde{a}_{13}=(39366\epsilon x^{25}+1220346\epsilon x^{24}+14565420\epsilon x^{23}+ \\ 83731482 \epsilon x^{22}-1003833 x^{23}+236309724 \epsilon x^{21}\\
-25095825 x^{22}  +708588 \epsilon^2 x^{19}
+299723976 \epsilon x^{20}\\-237858120 x^{21}+17714700 \epsilon^2x^{18} +191306610 \epsilon x^{19}\\-1080181170 x^{20}+143134776 \epsilon^2x^{17}+
140023890 \epsilon x^{18}\\-2462664465 x^{19}+462261816 \epsilon^2 x^{16}-236799612 \epsilon x^{17} \\-3116776563 x^{18}+553879620 \epsilon^2x^{15} +
98626896 \epsilon x^{16}\\-4040830962 x^{17}+384763284 \epsilon^2x^{14} -334491552 \epsilon x^{15}\\-5062097592 x^{16}+750333456 \epsilon^2 x^{13}+
149450184 \epsilon x^{14}\\-3027609900 x^{15}+681679152 \epsilon^2 x^{12} -469371996 \epsilon x^{13} \\-4012634700 x^{14}+460735776 \epsilon^2 x^{11}
-88155972 \epsilon x^{12}\\-1297170936 x^{13}+804260016 \epsilon^2x^{10} -332366796 \epsilon x^{11}\\-786840174 x^{12}+245153952 \epsilon^2 x^9
-90553302 \epsilon x^{10}\\-404146368 x^{11}+383522040 \epsilon^2x^8 -73322280 \epsilon x^9\\+451532556 x^{10} +111422304 \epsilon^2 x^7-
9084492 \epsilon x^8\\-104014116 x^9+74305512 \epsilon^2x^6-14703120 \epsilon x^7\\ +222024672 x^8+17635968 \epsilon^2 x^5-
7039224 \epsilon x^6\\-17230320 x^7-23184 \epsilon^2x^4-8030664 \epsilon x^5+88635312 x^6\\-414720 \epsilon^2 x^3-
3070548 \epsilon x^4-294176 x^5\\-1819584 \epsilon^2x^2+32403312x^4+419904 \epsilon^2 x+1692576 x^3\\+
944784 \epsilon^2+3895776 x^2+524880 x+944784)/(243 x^{11}\\+1944x^{10}+3627 x^9+1026 x^8+1944 x^7-8820 x^6\\+432 x^5-4860 x^4-72 x^3-652 x^2-72 x-324)^2 ;
\end{cases}$$
\end{itemize}
\end{example}
\section{Conclusion and Further Investigations}
\label{conclusion}
We proposed a Moser-based algorithm which recovers a singularly-perturbed linear differential system from its \textit{turning points} and reduces its $\epsilon$-rank to its minimal integer value.  A complementary step to attain the full formal reduction would be to find the ramification in the parameter which renders the general case to the case discussed here in a recursive process. One approach is that based on analysis by a Newton polygon and applied to system \eqref{particular} in \citep{key24} and the scalar case of system \eqref{general} in \citep{key77}. The sufficient number of coefficient matrices in computations is still to be investigated. 

In the usual treatment of \textit{turning points}, a \textit{restraining index} $\chi$ is defined and updated at every reduction step to observe the growth of the order of poles in the $A_k(x)$'s (see, e.g. \citep{key60,key59} and references therein). This \textit{restraining} index plays a role in the asymptotic interpretation of the formal solutions. As demonstrated in Section \ref{nested}, the transformations we apply to recover from turning points are polynomial (shearings). Hence, the growth of the poles order is bounded and can be expected apriori. The insight this gives into the \textit{restraining index} is to be investigated. 

Examples comparing this algorithm with a generalization of Levelt's favors the former. However, it suggests that Levelt's algorithm be generalized to system \eqref{general} and that a bit complexity study comparing both algorithms be held alongside. Furthermore, it motivates  generalization of the Moser-based algorithms over differential bivariate fields, e.g. Pfaffian systems in two variables \citep{key101}. 

An additional field of investigation is the two-parameter algebraic eigenvalue problem as a generalization of the one parameter case investigated via Moser-based approach in \citep{key54}. In fact, the main role in the reduction process is reserved to the similarity term of $T[A]$. Hence, the discussion of such problems is not expected to deviate from the discussion presented here in the differential case.

\bibliography{mybib}
\end{document}